\documentclass[a4paper,notitlepage]{article}
  
  \usepackage{latexsym}
  \usepackage{amsmath}
  \usepackage{amsfonts}
  \usepackage{cite}
  \usepackage{graphicx}
  \usepackage{color}
  
  \usepackage[T1]{fontenc}
  \usepackage[cp1250]{inputenc}

  \newtheorem{theorem}{Theorem}[section]

  \newtheorem{lemma}[theorem]{Lemma}
  \newtheorem{example}[theorem]{Example}
  
  \newenvironment{proof}[1][Proof]{\par\noindent\textbf{#1.}}{}
  \newcommand{\R}{\mathbb{R}}
  \newcommand{\N}{\mathbb{N}}

  \title{On the existence of homoclinic type solutions\\
  of inhomogenous Lagrangian systems}
  
  \author{\bf Jakub Ciesielski, Joanna Janczewska\\
  \small Faculty of Applied Physics and Mathematics\\
  \small Gda\'{n}sk University of Technology\\
  \small Narutowicza 11/12, 80-233 Gda\'{n}sk, Poland\\
  \small jciesielski@mif.pg.gda.pl, janczewska@mif.pg.gda.pl\\
  \bf Nils Waterstraat\\
  \small School of Mathematics, Statistics and Actuarial Science\\
  \small University of Kent, Canterbury\\
  \small Kent CT2 7NF, England\\
  \small N.Waterstraat@kent.ac.uk\\}
  
\begin{document}
  
    \maketitle
    
    \begin{abstract}
    
    We study the existence of homoclinic type solutions for second order
    Lagrangian systems of the type $\ddot{q}(t)-q(t)+a(t)\nabla G(q(t))=f(t)$,
    where $t\in\R$, $q\in\R^n$, $a\colon\R\to\R$ is a continuous positive
    bounded function, $G\colon\R^n\to\R$ is a $C^1$-smooth potential
    satisfying the Ambrosetti-Rabinowitz superquadratic growth condition
    and $f\colon\R\to\R^n$ is a continuous bounded square integrable forcing term.
    A homoclinic type solution is obtained as limit of $2k$-periodic solutions
    of an approximative sequence of second order differential equations.
    
    \end{abstract}
    
    \textbf{key words:} Homoclinic type solution; Lagrangian system;
    Critical point
    
    
    \textbf{running head:} On the existence of homoclinic type solutions
             
    
    \section{Introduction}
    
    The aim of this paper is to prove the existence of a solution
    for the second order Lagrangian system
    
    \begin{equation}\label{LS}
      \begin{cases}
        \ddot{q}(t)-q(t)+a(t)\nabla G(q(t))=f(t),\\
        \lim_{t\to\pm\infty}q(t)=\lim_{t\to\pm\infty}\dot{q}(t)=0,
      \end{cases}
    \end{equation}
    where $a\colon\R\to\R$ is a continuous positive bounded function,
    $G\colon\R^n\to\R$, $n\geq 1$, is a $C^1$-smooth potential
    satisfying the Ambrosetti-Rabinowitz superquadratic growth condition
    and $f\colon\R\to\R^n$ is a continuous bounded square integrable forcing term.
    
    Our intention is to generalise the following result by E. Serra, M. Tarallo
    and S. Terracini from \cite{STT} to the inhomogeneous systems \eqref{LS}.
    
    \begin{theorem}\label{STTthm}
      Assume that
      \begin{itemize}
        \item[$(\tilde{C}1)$] $G\in C^2(\R^n,\R)$,
        \item[$(C2)$] there exists $\mu>2$ such that for all $x\in\R^n\setminus\{0\}$,
        \[0<\mu G(x)\leq(\nabla G(x),x),\]
        \item[$(\tilde{C}3)$] $a\in C(\R,\R)$ is almost periodic in the sense of Bohr
        and \[\inf_{t\in\R}a(t)>0.\]
      \end{itemize}
      Then the problem
      \begin{equation}
        \begin{cases}
          \ddot{q}(t)-q(t)+a(t)\nabla G(q(t))=0,\\
          \lim_{t\to\pm\infty}q(t)=\lim_{t\to\pm\infty}\dot{q}(t)=0
        \end{cases}
      \end{equation}
      has at least one nonzero solution.
    \end{theorem}
    
    Here and subsequently, $(\cdot,\cdot)\colon\R^n\times\R^n\to\R$
    denotes the standard inner product in $\R^n$, and $|\cdot|\colon\R^n\to[0,\infty)$
    is the induced norm. Let us recall that a function $a$
    is almost periodic in the sense of Bohr if for every $\varepsilon>0$
    there is a finite linear combination of sine and cosine functions
    that is of distance less than $\varepsilon$ from $a$ with respect to
    the supremum norm.
    
    The proof of Theorem \ref{STTthm} in \cite{STT} is of variational nature,
    i.e. a solution is found as a critical point of a suitable functional. 
    The lack of a group of symmetries for which the functional is invariant,
    which exists in the case of periodic potentials, is faced by a property
    of Palais-Smale sequences introduced by E. S\'{e}r\'{e}
    (see \cite{S}) and Bochner's criterion of almost periodicity (see \cite{B}).
    
    Let us now consider the inhomogeneous Lagrangian systems \eqref{LS}.
    Intuitively, if the forcing term $f(t)$ in \eqref{LS} is sufficiently small,
    then a homoclinic type solution should exist simply because of the existence
    in the homogenous case.
    
    Our main result affirms this and it also deals with the question
    how large the forcing term $f$ in \eqref{LS} can be:
    
    \begin{theorem}\label{CJWthm}
      Assume that
      \begin{itemize}
        \item[$(C1)$] $G\in C^1(\R^n,\R)$ and $|\nabla G(q)|=o(|q|)$ as $|q|\to 0$,
        \item[$(C2)$] there exists $\mu>2$ such that for all $x\in\R^n\setminus\{0\}$,
        \[0<\mu G(x)\leq(\nabla G(x),x),\]
        \item[$(C3)$] $a\in C(\R,\R)$ and $\inf_{t\in\R}a(t)>0$,
        \item[$(C4)$] $M:=\sup\{a(t)G(x)\colon t\in\R,\ |x|=1\}<\frac{1}{2}$,
        \item[$(C5)$] $\left(\int_{\R}|f(t)|^2dt\right)^{\frac{1}{2}}<\frac{1}{2\sqrt{2}}(1-2M)$.  
      \end{itemize}
      Then the inhomogenous Lagrangian system \eqref{LS} has at least one
      homoclinic type solution.
    \end{theorem}
    
    Let us briefly discuss our assumptions in Theorem \ref{CJWthm}.
    Condition $(C2)$ is the superquadratic growth condition
    due to A. Ambrosetti and P. Rabinowitz \cite{AmR}.
    Since $G$ has a global minimum at $0$ by $(C2)$, $(C1)$ is more general than $(\tilde{C}1)$.
    Moreover, it is readily seen by $(C2)$ that for every $q\neq 0$ the map
    \begin{displaymath}
      (0,\infty)\ni \xi\mapsto G(\xi^{-1} q)|\xi|^\mu\in\R
    \end{displaymath}
    is non-increasing, which yields the following inequalities:
    \begin{equation}\label{Amb}
      G(q)\leq G\left(\frac{q}{|q|}\right)|q|^{\mu},\ \textrm{if} \ 0<|q|\leq 1
    \end{equation}
    and
    \begin{equation}\label{Rab}
      G(q)\geq G\left(\frac{q}{|q|}\right)|q|^{\mu},\ \textrm{if} \ |q|\geq 1.
    \end{equation}
    As $\mu>2$, the inequality \eqref{Rab} implies that $G$
    grows faster than $|\cdot|^2$ at infinity.
    
    Clearly, $(C3)$ is more general than $(\tilde{C}3)$. Note that $(C3)$ and $(C4)$
    imply that $a$ is bounded, which, however, is also true for every
    almost periodic function in the sense of Bohr.
    
    The last two conditions $(C4)$ and $(C5)$ are closely related.
    Namely, the forcing term $f$ needs to be sufficiently small in $L^2(\R,\R^n)$,
    but the upper bound on the norm of $f$ depends on the restriction
    of the space variable of the potential $a(t)G(x)$ to the unit sphere in $\R^n$.
    
    The study of homoclinic solutions for Lagrangian systems has received
    much attention in recent years, especially when the potential is periodic in time.
    The existence problem of homoclinics has been widely investigated
    by variational methods, see for example in \cite{ArSz, BeM, CZ, R1, R2, RaT, S}.
    Existence results for perturbed systems were given in \cite{IzJ, IzJ2, J1, J2, K, Sal}.
    
    Our proof of Theorem \ref{CJWthm} is also of variational nature.
    Let us point out, however, that it is quite different from Serra,
    Tarallo and Terracini's proof of Theorem \ref{STTthm} in \cite{STT}.
    Here, we find a solution of \eqref{LS} as a limit in $C^{2}_{loc}(\R,\R^n)$
    of a sequence $\{q_k\}_{k\in\N}$, $q_{k}\in W^{1,2}_{2k}(\R,\R^n)$,
    obtained by an approximation scheme introduced by Krawczyk in \cite{K},
    where every $q_k$ is a critical point of a suitable functional $I_k$
    that we introduce below.
    
    We now prove Theorem \ref{CJWthm} in the following section
    and we nicely round off the paper by two numerical examples in a final section.

    
    \section{Proof of Theorem \ref{CJWthm}}
    
    In what follows, we let $E$ be the Sobolev space $W^{1,2}(\R,\R^n)$
    of $W^{1,2}$-functions on $\R$ with values in $\R^n$ equipped with the norm
    
    \begin{displaymath}
      \|q\|_E=\left(\int_{-\infty}^{\infty}\left(|q(t)|^2
      +|\dot{q}(t)|^2\right)dt\right)^\frac{1}{2}.
    \end{displaymath}
    For each $k\in\N$, we denote by $E_k=W_{2k}^{1,2}(\R,\R^n)$
    the Sobolev space of $2k$-periodic $W^{1,2}$-functions with the norm
    
    \begin{displaymath}
      \|q\|_{E_k}=\left(\int_{-k}^{k}\left(|q(t)|^2
      +|\dot{q}(t)|^2\right)dt\right)^\frac{1}{2}.
    \end{displaymath}
    Then let $L^{\infty}_{2k}(\R,\R^n)$ be the space of $2k$-periodic,
    essentialy bounded and measureable functions from $\R$ into $\R^n$ with the norm
    
    \begin{displaymath}
      \|q\|_{L^{\infty}_{2k}}=\textrm{ess}\sup\{|q(t)|\colon t\in[-k,k]\}.
    \end{displaymath}
    We note for later reference that
      
    \begin{equation}\label{constant}
      \|q\|_{L^{\infty}_{2k}}\leq\sqrt{2}\|q\|_{E_k}
    \end{equation}
    for all $k\in\N$ and $q\in E_k$ (cf. \cite[Fact 2.8]{IzJ}).
    Finally, let $C_{loc}^2(\R,\R^n)$ denote the space
    of $C^2$-functions with the topology of almost uniform convergence
    of functions and all their derivatives up to second order.
    
    The following result can be found in \cite[Thm. 1.3]{K}.
    
    \begin{theorem}\label{Kthm}
      Let $f\colon\R\to\R^n$ be a non-trivial, bo\-unded, continuous
      and square integrable map and $V\colon\R\times\R^n\to\R$
      a $C^{1}$-smooth potential such that $\nabla_{q}V\colon\R\times\R^n\to\R^n$
      is bounded in the time variable.
      Assume that for each $k\in\N$ the boundary value problem
      \begin{equation}\label{hsk}
        \left\{
          \begin{array}{ll}
            \ddot{q}(t)+\nabla_{q}V_k(t,q(t))=f_k(t),\\
            q(-k)-q(k)=\dot{q}(-k)-\dot{q}(k)=0, 
          \end{array}
        \right.
      \end{equation}
      where $f_k\colon\R\rightarrow\R$ is a $2k$-periodic extension of $f\mid_{[-k,k)}$ 
      and $V_k\colon\R\times\R^n\to\R$ is a $2k$-periodic extension of $V\mid_{[-k,k)\times\R^n}$,
      has a periodic solution $q_k\in E_k$ and $\{\|q_k\|_{E_k}\}_{k\in\N}$
      is a bounded sequence in $\R$.
      Then there exists a subsequence $\{q_{k_j}\}_{j\in\N}$
      converging in the topology of $C_{loc}^2(\R,\R^n)$ to a function $q\in E$
      which is a homoclinic type solution of the Newtonian system
      \begin{equation}\label{hs}
        \ddot{q}(t)+\nabla_{q}V(t,q(t))=f(t),\ \ t\in\R. 
      \end{equation}
    \end{theorem}
  
    Our aim is to obtain a homoclinic type solution of \eqref{LS} by Theorem \ref{Kthm}
    as a limit in $C_{loc}^2(\R,\R^n)$ of a sequence $\{q_k\}_{k\in N}$
    such that for each $k\in\N$, $q_k\in E_k$ is a $2k$-periodic solution
    of the boundary value problem
    
    \begin{equation}\label{lsk}
    \left\{
      \begin{array}{ll}
        \ddot{q}(t)-q(t)+a_k(t)\nabla G(q(t))=f_k(t),\\
        q(-k)-q(k)=\dot{q}(-k)-\dot{q}(k)=0, 
      \end{array}
      \right.
    \end{equation}
    where $f_k$ is as above and $a_k\colon\R\to\R$ is a $2k$-periodic extension
    of $a\mid_{[-k,k)}$.
    
    To this purpose, we now define for $k\in\N$ a functional $I_k\colon E_k\to\R$ by
    
    \begin{equation}\label{Ik}
      I_k(q)=\frac{1}{2}\left\|q\right\|_{E_k}^2-\int_{-k}^{k}a_k(t)G(q(t))dt
      +\int_{-k}^{k}(f_k(t),q(t))dt. 
    \end{equation}
    Then $I_k\in C^1(E_k,\R)$ and, moreover,
    
    \begin{align}\label{derivative}
    \begin{split}
      I'_k(q)v=&\int_{-k}^{k}(\dot{q}(t),\dot{v}(t))dt
      +\int_{-k}^{k}(q(t),v(t))dt\\
      &-\int_{-k}^{k}(a_k(t)\nabla G(q(t)),v(t))dt
      +\int_{-k}^{k}(f_k(t),v(t))dt.
    \end{split}
    \end{align}
    Hence
    
    \begin{equation}\label{prime}
      I'_k(q)q=\|q\|_{E_k}^2-\int_{-k}^{k}(a_k(t)\nabla G(q(t)),q(t))dt
      +\int_{-k}^{k}(f_k(t),q(t))dt.
    \end{equation}
    Let us note for later reference that by \cite[Fact 2.2]{IzJ}, if
    
    \begin{displaymath}
      m:=\inf\{a(t)G(x)\colon t\in\R,\ |x|=1\}
    \end{displaymath}
    and $\mu>2$ is defined as in Theorem \ref{CJWthm},
    then for all $\zeta\in\R\setminus\{0\}$ and $q\in E_k\setminus\left\{0\right\}$
    
    \begin{equation}\label{zeta}
      \int_{-k}^{k}a_k(t)G(\zeta q(t))dt\geq m|\zeta|^{\mu}\int_{-k}^{k}|q(t)|^{\mu}dt-2km.
    \end{equation}
    Clearly, critical points of the functional $I_k$ are classical
    $2k$-periodic solutions of \eqref{lsk}. We will now obtain a critical point
    of $I_k$ by using the Mountain Pass Theorem from \cite{AmR}. This theorem provides
    the minimax characterisation for a critical value which is important
    for our argument. Let us recall its statement for the convenience of the reader.
    
    \begin{theorem}\label{MPT}
      Let $E$ be a real Banach space and $I\colon E\rightarrow\R$
      a $C^1$-smooth functional. If $I$ satisfies the following conditions:
      \begin{itemize}
        \item[$(i)$] $I(0)=0$,
        \item[$(ii)$] every sequence $\{u_j\}_{j\in\N}\subset E$
        such that $\{I(u_j)\}_{j\in\N}$ is bounded in $\R$
        and $I'(u_j)\rightarrow 0$ in $E^*$ as $j\rightarrow\infty$
        contains a convergent subsequence (Palais-Smale condition),
        \item[$(iii)$] there exist constants $\rho,\alpha>0$
        such that $I|_{\partial B_{\rho}(0)}\geq\alpha$,
        \item[$(iv)$] there is some $e\in E\setminus\overline{B}_{\rho}(0)$
        such that $I(e)\leq 0$,
      \end{itemize}
      where $B_{\rho}(0)$ denotes the open ball in $E$ of radius $\rho$ about $0$,
      then $I$ has a critical value $c\geq\alpha$ given by
      \[c=\inf_{g\in\Gamma}\max_{s\in[0,1]}I(g(s)),\]
      where \[\Gamma:=\left\{g\in C([0,1],E)\colon g(0)=0,\ g(1)=e \right\}.\]
    \end{theorem}
    
    The following lemma, in combination with Theorem \ref{Kthm},
    is the keystone of our proof of Theorem \ref{CJWthm}.
    
    \begin{lemma}\label{geometry}
      For each $k\in\N$, the functional $I_k$ given by \eqref{Ik}
      has the mountain pass geometry, i.e. it satisfies all assumptions
      of Theorem \ref{MPT}.
    \end{lemma}
    
    \begin{proof}\ 
      We fix $k\in\N$ and we now show the assumptions (i)-(iv) in Theorem \ref{MPT} for $I_k$.
      It is clear that $I_k(0)=0$, which is (i). In order to show the Palais-Smale condition (ii),
      we consider a sequence $\{u_j\}_{j\in\N}\subset E_k$
      such that $\{I_k(u_j)\}_{j\in\N}\subset\R$ is bounded
      and $I'_k(u_j)\rightarrow 0$ in $E^{*}_{k}$ as $j\rightarrow +\infty$.
      Consequently, there exists a constant $C_k>0$ such that for all $j\in\N$
      we have
      
      \begin{equation}\label{PS}
        |I_k(u_j)|\leq C_k,\ \ \ \|I'_k(u_j)\|_{E_{k}^{*}}\leq C_k.
      \end{equation}
      
      By \eqref{Ik} and $(C2)$ we get
      
      \begin{align}\label{boundedness}
      \begin{split}
        \|u_j\|_{E_k}^2\leq & 2I_k(u_j)-2\int_{-k}^{k}(f_k(t),u_j(t))dt\\
        & +\frac{2}{\mu}\int_{-k}^{k}(a_k(t)\nabla G(u_j(t)),u_j(t))dt.
      \end{split}
      \end{align}
      As
      
      \begin{align*}
        \int^k_{-k}(a_k(t)\nabla G(u_j(t)),u_j(t))dt
        =\|u_{j}\|^2_{E_k}-I'_k(u_j)u_j+\int^k_{-k}(f_k(t),u_j(t))dt
      \end{align*}
      by \eqref{prime}, we obtain      
      
      \begin{align*}
        \left(1-\frac{2}{\mu}\right)\|u_j\|^2_{E_k}
        &\leq 2 I_k(u_j)-\frac{2}{\mu}I'_k(u_j)(u_j)
        -2\int^k_{-k}(f_k(t),u_j(t))dt\\
        &+\frac{2}{\mu}\int^k_{-k}(f_k(t),u_j(t))dt
      \end{align*}     
      and so
      
      \begin{equation}\label{estimation}
      \begin{split}
        \left(1-\frac{2}{\mu}\right)\|u_j\|_{E_k}^2\leq & 2I_k(u_j)\\
        & +\left(\frac{2}{\mu}\|I'_k(u_j)\|_{E^{*}_{k}}
        +\left(2-\frac{2}{\mu}\right)\|f_k\|_{L^2_{2k}}\right)\|u_j\|_{E_k},
      \end{split}
      \end{equation}
      where we denote by
      
      \begin{align*}
        \|q\|_{L^2_{2k}}=\left(\int^k_{-k}|q(t)|^2dt\right)^{\frac{1}{2}}
      \end{align*}
      the norm of the space $L^2_{2k}(\R,\R^n)$ of all $2k$-periodic $L^2$-functions.
      Combining \eqref{estimation} with $(C5)$ and \eqref{PS} we get
      
      \begin{displaymath}
        \left(1-\frac{2}{\mu}\right)\|u_j\|_{E_k}^2
        -\left(\frac{2C_k}{\mu}
        +\frac{1}{\sqrt{2}}\left(1-\frac{1}{\mu}\right)(1-2M)\right)\|u_j\|_{E_k}
        -2C_k\leq 0,
      \end{displaymath}
      which yields the boundedness of $\{u_j\}_{j\in\N}$ in $E_k$ as $\mu>2$ by $(C2)$.
      Going to a subsequence if necessary, we can assume that there exists a function
      $u\in E_k$ such that $u_j\rightharpoonup u$ weakly in $E_k$ as $j\rightarrow +\infty$,
      and hence $\{u_j\}_{j\in\N}$ also converges to $u$ uniformly, as $E_k$
      is compactly embedded in $C([-k,k],\R^n)$.
      This shows in particular that $\|u_{j}-u\|_{L^{2}_{2k}}\rightarrow 0$
      as $j\rightarrow\infty$.
      
      Applying \eqref{derivative} we have
      
      \begin{align*}
        I'_k(u_j)&(u_j-u)=\int^k_{-k}(\dot{u}_j(t),\dot{u}_j(t)-\dot{u}(t))dt
        +\int^k_{-k}(u_j(t),u_j(t)-u(t))dt\\
        &-\int^k_{-k}(a_k(t)\nabla G(u_j(t)),u_j(t)-u(t))dt
        +\int^k_{-k}(f_k(t),u_j(t)-u(t))dt
      \end{align*}
      and
      
      \begin{align*}
      I'_k(u)&(u_j-u)=\int^k_{-k}(\dot{u}(t),\dot{u}_j(t)-\dot{u}(t))dt
      +\int^k_{-k}(u(t),u_j(t)-u(t))dt\\
      &-\int^k_{-k}(a_k(t)\nabla G(u(t)),u_j(t)-u(t))dt
      +\int^k_{-k}(f_k(t),u_j(t)-u(t))dt,
      \end{align*}
      which yields
      
      \begin{align*}
      \|\dot{u}_j-\dot{u}\|^2_{L^2_{2k}}&=(I'_k(u_j)-I'_k(u))(u_j-u)-\|u_j-u\|^2_{L^2_{2k}}\\
      &+\int^k_{-k}a_k(t)(\nabla G(u_j(t))-\nabla G(u(t)),u_j(t)-u(t))dt.
      \end{align*}
      As $I'_k(u_j)$ is bounded by \eqref{PS}, $\nabla G$ is continuous
      and $u_j\rightarrow u$ uniformly, we see that
      $\|\dot{u}_j-\dot{u}\|^2_{L^2_{2k}}\rightarrow 0$.
      Hence $\|\dot{u}_j-\dot{u}\|_{E_k}\rightarrow 0$ and the Palais-Smale condition is shown.
      
      In the next step we will prove that there exist constants $\rho>0$ and $\alpha>0$
      independent of $k\in\N$ such that ${I_{k}}|_{\partial B_{\rho}(0)}\geq\alpha$, which is (iii).
      Assume that $0<\|q\|_{L^{\infty}_{2k}}\leq 1$. By \eqref{Amb} and $(C4)$ we obtain
      
      \begin{displaymath}
        \int_{-k}^{k}a_k(t)G(q(t))dt\leq\int_{-k}^{k}a_k(t)G\left(\frac{q(t)}{|q(t)|}\right)|q(t)|^{\mu}dt
        \leq M\int_{-k}^{k}|q(t)|^2dt\leq M\|q\|^2_{E_k}.
      \end{displaymath}      
      Combining this with \eqref{Ik} we get
      
      \begin{displaymath}
        I_k(q)\geq\frac{1}{2}(1-2M)\|q\|^2_{E_k}-\|f\|_{L^2}\|q\|_{E_k}.
      \end{displaymath}
      Let $\rho=\frac{1}{\sqrt{2}}$
      and $\alpha=\frac{1}{\sqrt{2}}\left(\frac{1}{2\sqrt{2}}(1-2M)-\|f\|_{L^2}\right)$.
      From $(C5)$ it follows that $\alpha>0$.
      Using \eqref{constant}, if $\|q\|_{E_k}=\rho$, then $0<\|q\|_{L^{\infty}_{2k}}\leq 1$,
      which implies $I_k(q)\geq\alpha$.
      
      It remains to prove (iv), i.e. that for all $k\in\N$ there is $e_k\in E_k$
      such that $\|e_k\|_{E_k}>\rho$ and $I_k(e_k)\leq 0$.
      Applying \eqref{Ik} and \eqref{zeta}, we have
      
      \begin{displaymath}
        I_k(\zeta q)\leq\frac{\zeta^2}{2}\|q\|^{2}_{E_k}
        -m|\zeta|^{\mu}\int_{-k}^{k}|q(t)|^{\mu}dt
        +|\zeta|\|f_k\|_{L^{2}_{2k}}\|q\|_{E_k}+2km
      \end{displaymath}
      for all $\zeta\in\R\setminus\{0\}$ and $q\in E_k\setminus\{0\}$.
      Let us choose $0\neq Q\in E_1$ such that $Q(\pm 1)=0$.
      As $\mu>2$ and $m>0$, there exists $\zeta\in\R\setminus\{0\}$
      such that $\|\zeta Q\|_{E_1}>\rho$ and $I_1(\zeta Q)<0$.
      
      We set $e_1(t)=\zeta Q(t)$ and define for each positive integer $k>0$,
      
      \begin{displaymath}
        e_k(t)=
        \begin{cases}
          e_1(t)\ \textrm{if}\ |t|\leq 1\\
          0\ \textrm{if}\ 1<|t|\leq k.
        \end{cases}
      \end{displaymath}
      Then $e_k\in E_k$, $\|e_k\|_{E_k}=\|e_1\|_{E_1}>\rho$ and $I_k(e_k)=I_1(e_1)<0$
      for all $k\in\N$.
      
      Consequently, by Theorem \ref{MPT}, the action $I_k$ has a critical value
      $c_k\geq\alpha$ given by
      
      \begin{equation}\label{critpoint}
        c_k=\inf_{g\in\Gamma_k}\max_{s\in[0,1]}I_k(g(s)),
      \end{equation}
      where $\Gamma_k=\{g\in C([0,1],E_k)\colon g(0)=0,\ g(1)=e_k\}$.
            
    \end{proof}
    
    \hspace{10cm} $\Box$
    
    From Lemma \ref{geometry} we conclude that for all $k\in\N$
    there exists $q_k\in E_k$ such that
    
    \begin{displaymath}
      I_k(q_k)=c_k,\ \ \ I'_k(q_k)=0,
    \end{displaymath}
    where the values $c_k$ are given by \eqref{critpoint}.
    By Theorem \ref{Kthm}, in order to finish the proof of Theorem \ref{CJWthm},
    it suffices to show that the sequence of real numbers $\{\|q_k\|_{E_k}\}_{k\in\N}$
    is bounded. For this purpose, set
    
    \begin{displaymath}
      M_0=\max_{s\in[0,1]}I_1(se_{1}).
    \end{displaymath}
    As $I_k(se_{k})=I_1(se_{1})$ for all $s\in[0,1]$, $k\in\N$, and $se_k\in\Gamma_k$,
    we have by \eqref{critpoint}
    
    \begin{equation}\label{bound}
      c_k=\inf_{g\in\Gamma_k}\max_{s\in[0,1]}I_k(g(s))\leq M_0.
    \end{equation}
    Using \eqref{Ik} and \eqref{prime} we obtain
    
    \begin{align*}
      c_k=I_k(q_k)-\frac{1}{2}I'_{k}(q_k)q_k=&
      \frac{1}{2}\int_{-k}^{k}\left(a_k(t)\nabla G(q_k(t)),q_k(t)\right)dt\\
      &-\int_{-k}^{k}a_k(t)G(q_k(t))dt
      +\frac{1}{2}\int_{-k}^{k}(f_k(t),q_k(t))dt
    \end{align*}
    for all $k\in\N$, and by $(C2)$,
    
    \begin{displaymath}
      c_k\geq\left(\frac{\mu}{2}-1\right)\int_{-k}^{k}a_k(t)G(q_k(t))dt
      -\frac{1}{2}\|f_k\|_{L^{2}_{2k}}\|q_k\|_{E_k}.
    \end{displaymath}
    Hence
    
    \begin{displaymath}
      \int_{-k}^{k}a_k(t)G(q_k(t))dt
      \leq\frac{1}{\mu-2}\left(2c_{k}+\|f_k\|_{L^{2}_{2k}}\|q_k\|_{E_k}\right).
    \end{displaymath}
    Combining this with \eqref{Ik}, \eqref{bound} and $(C5)$,
    for each $k\in\N$, we have
    
    \begin{align*}
      c_k&=I_k(q_k)=\frac{1}{2}\|q_k\|^2_{E_k}-\int^k_{-k}a(t)G(q_k(t))dt
      +\int^k_{-k}(f_k(t),q_k(t))dt\\
      &\geq \frac{1}{2}\|q_k\|^2_{E_k}
      -\frac{1}{\mu-2}\left(2c_{k}+\|f_k\|_{L^{2}_{2k}}\|q_k\|_{E_k}\right)
      -\|f_k\|_{L^2_{2k}}\|q_k\|_{E_k},
    \end{align*}
    and so
    
    \begin{displaymath}
      \|q_k\|_{E_k}^{2}
      -\frac{1}{\sqrt{2}}\frac{\mu-1}{\mu-2}\left(1-2M\right)\|q_k\|_{E_k}
      -\frac{2\mu M_{0}}{\mu-2}\leq 0,
    \end{displaymath}
    which completes the proof of Theorem \ref{CJWthm}.

    
    \section{One-dimensional Examples}
    
    In this section we present two simple one-dimensional examples,
    i.e. we consider the case $n=1$.
    
    \begin{example}\label{ex1}
      Let a function $a\colon\R\to\R$, a forcing term $f\colon\R\to\R$
      and a potential $G\colon\R\to\R$ be given as follows:
      
      \begin{align*}
        a(t)&=\frac{1}{5}e^{-t^2}+\frac{1}{10},\ \ \ t\in\R,\\
        f(t)&=\frac{2}{5}e^{-\frac{t^2}{2}},\ \ \ t\in\R,\\
        G(q)&=q^4,\ \ \ q\in\R.
      \end{align*}
      It is easy to check that $a$, $f$ and $G$ satisfy the assumptions $(C1)-(C5)$.
      
      The figures \ref{fig:1k10}-\ref{fig:1k200} show the graphs of approximative
      solutions $q_k$ of \eqref{lsk} for $k=10,16,90,140,200$.
    \end{example}
      
    \begin{figure}
      \centering
      \includegraphics[width=10cm]{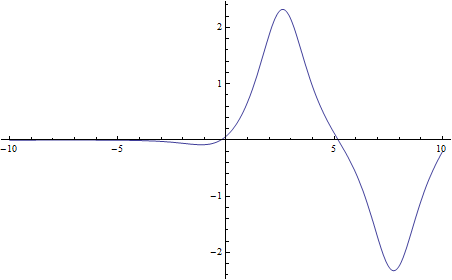}
      \caption{Example \ref{ex1}, an approximative solution of \eqref{lsk} for $k=10$}
      \label{fig:1k10}
    \end{figure}
      
    \begin{figure}\label{fig:1k16}
      \centering
      \includegraphics[width=10cm]{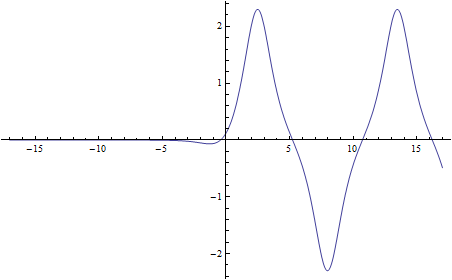}
      \caption{Example \ref{ex1}, an approximative solution of \eqref{lsk} for $k=16$}
    \end{figure}
      
    \begin{figure}
      \centering
      \includegraphics[width=10cm]{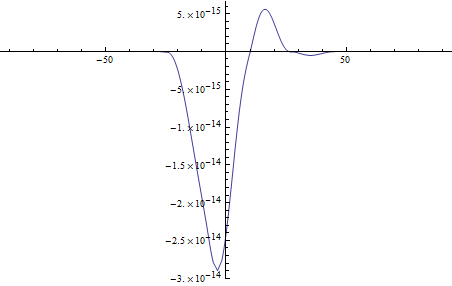}
      \caption{Example \ref{ex1}, an approximative solution of \eqref{lsk} for $k=90$}
      \label{fig:1k90}
    \end{figure}
      
    \begin{figure}
      \centering
      \includegraphics[width=10cm]{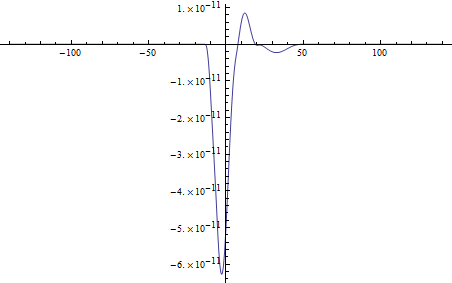}
      \caption{Example \ref{ex1}, an approximative solution of \eqref{lsk} for $k=140$}
      \label{fig:1k140}
    \end{figure}
      
    \begin{figure}
      \centering
      \includegraphics[width=10cm]{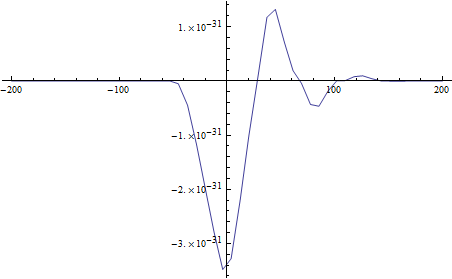}
      \caption{Example \ref{ex1}, an approximative solution of \eqref{lsk} for $k=200$}
      \label{fig:1k200}
    \end{figure}
      
    
    \begin{example}\label{ex2}
      Let us define functions $a,f,G\colon\R\to\R$ as follows:
      
      \begin{align*}
        a(t)&=\frac{1}{\pi}\arctan(t)+\frac{1}{2},\ \ t\in\R,\\
        f(t)&=\frac{1}{2}e^{-\frac{t^2}{2}},\ \ t\in\R,\\
        G(q)&=q^4,\ \ q\in\R.
      \end{align*}
      Again, it is readily seen that the assumptions $(C1)-(C5)$ are satisfied.
      
      The figures \ref{fig:2k10}-\ref{fig:2k140} show the graphs of approximative
      solutions $q_k$ of \eqref{lsk} for $k=10,16,90,140$.
    \end{example}
   
    \begin{figure}
      \centering
      \includegraphics[width=10cm]{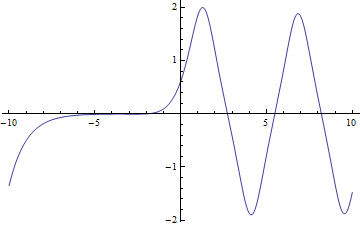}
      \caption{Example \ref{ex2}, an approximative solution of \eqref{lsk} for $k=10$}
      \label{fig:2k10}
    \end{figure}
    
    \begin{figure}
      \centering
      \includegraphics[width=10cm]{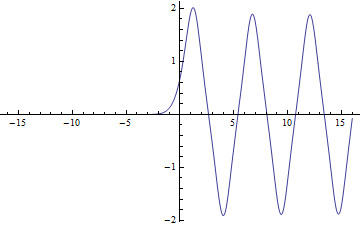}
      \caption{Example \ref{ex2}, an approximative solution of \eqref{lsk} for $k=16$}
      \label{fig:2k16}
    \end{figure}
    
    \begin{figure}
      \centering
      \includegraphics[width=10cm]{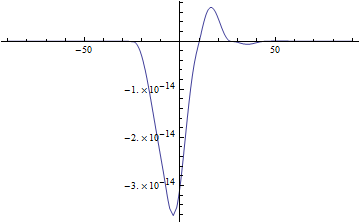}
      \caption{Example \ref{ex2}, an approximative solution of \eqref{lsk} for $k=90$}
      \label{fig:2k90}
    \end{figure}
    
    \begin{figure}
      \centering
      \includegraphics[width=10cm]{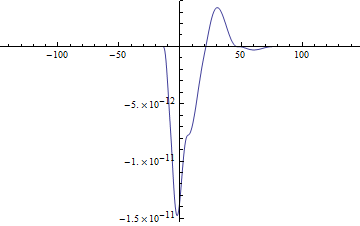}
      \caption{Example \ref{ex2}, an approximative solution of \eqref{lsk} for $k=140$}
      \label{fig:2k140}
    \end{figure}
    
    
    \subsection*{Acknowledgments}
    The second and the third author are supported by Grant PPP-PL no. 57217076
    of DAAD and MNiSW.

    
    \bibstyle

  \end{document}